\documentclass{article}

\usepackage{amsmath,amssymb,amsthm}

\newtheorem{thm}{Theorem}
\newtheorem{cor}{Corollary}

\theoremstyle{remark}
\newtheorem{rem}{Remark}

\theoremstyle{definition}

\title{Generalized retarded integral inequalities}

\author{Rui A. C. Ferreira\footnote{Supported by FCT through the PhD fellowship SFRH/BD/39816/2007.}\\
\texttt{ruiacferreira@ua.pt}
\and Delfim F. M. Torres\footnote{Supported by FCT through the R\&D unit CEOC, cofinanced by the EC fund FEDER/POCI 2010.}\\
\texttt{delfim@ua.pt}}

\date{Department of Mathematics\\
University of Aveiro\\
3810-193 Aveiro, Portugal}

\begin{document}

\maketitle

\begin{abstract}
We prove some new retarded integral inequalities. The results
generalize those in [J. Math. Anal. Appl. 301 (2005), no. 2,
265--275].

\bigskip

\noindent \textbf{Keywords:} retarded inequalities,
integral inequalities, retarded differential equations,
global existence, time delay.

\bigskip

\noindent \textbf{2000 Mathematics Subject Classification:} 26D15,
26D20.

\end{abstract}


\section{Introduction}

In \cite{lip} O.~Lipovan proved some retarded versions of the
inequalities of Ou-Iang \cite{ou} and Pachpatte \cite{Pach}. From
Lipovan's retarded inequalities, some criteria are obtained
ensuring the global existence of solutions to the generalized
Li{\'e}nard equation with time delay, and to a retarded Rayleigh
type equation \cite{lip}. More recently, Y.G.~Sun generalized the
Lipovan's retarded inequalities \cite{yuan}. In this note we prove two theorems that give as corollaries all the results obtained in
the above cited papers. Our results can be used to further study
the global existence of solutions to differential equations with a time delay. Several generalizations in different directions than
that followed here are found in the literature (see \textrm{e.g.}
\cite{agar,Jiang,lip06,Ma,Ma07,Ma08,Xu,Zhao}).


\section{Main results}
\label{sec:mainResults}

Throughout we use the notation $\mathbb{R}_0^+=[0,\infty)$,
$\mathbb{R}^+=(0,\infty)$, and $Dom(f)$ and Im$(f)$ to denote, respectively, the domain and the image of a function $f$.

\begin{thm}
\label{teor1} Let $f(t,s)$ and $g(t,s)$ $\in
C(\mathbb{R}_0^+\times\mathbb{R}_0^+,\mathbb{R}_0^+)$ be
nondecreasing in $t$ for every $s$ fixed. Moreover, let $\phi\in
C(\mathbb{R}_0^+,\mathbb{R}_0^+)$ be a strictly increasing
function such that $\lim_{x\rightarrow\infty}\phi(x)=\infty$ and
suppose that $c\in C(\mathbb{R}_0^+,\mathbb{R}^+)$ is a
nondecreasing function. Further, let $\eta, w\in
C(\mathbb{R}_0^+,\mathbb{R}_0^+)$ be nondecreasing with
$\{\eta,w\}(x)>0$ for $x\in(0,\infty)$ and
$\int_{x_0}^\infty\frac{ds}{\eta(\phi^{-1}(s))}=\infty$, with
$x_0$ defined as below. Finally, assume that $\alpha\in
C^1(\mathbb{R}_0^+,\mathbb{R}_0^+)$ is nondecreasing with
$\alpha(t)\leq t$. If $u\in C(\mathbb{R}_0^+,\mathbb{R}_0^+)$
satisfies
\begin{equation}
\label{in1} \phi(u(t))\leq
c(t)+\int_0^{\alpha(t)}\left[f(t,s)\eta(u(s))w(u(s))+g(t,s)\eta(u(s))\right]ds,\
t\in\mathbb{R}_0^+,
\end{equation}
then there exists $\tau\in\mathbb{R}^+$ such that, for all
$t\in[0,\tau]$, we have
$$\Psi(p(t))+\int_0^{\alpha(t)}f(t,s)ds\in Dom(\Psi^{-1}),$$
and
\begin{equation}
\label{in2} u(t)\leq
\phi^{-1}\left\{G^{-1}\left(\Psi^{-1}\left[\Psi(p(t))
+\int_0^{\alpha(t)}f(t,s)ds\right]\right)\right\},
\end{equation}
where
$$G(x)=\int_{x_0}^x\frac{ds}{\eta(\phi^{-1}(s))},$$
with $x\geq c(0)>x_0>0$ if
$\int_{0}^x\frac{ds}{\eta(\phi^{-1}(s))}=\infty$ and $x\geq
c(0)>x_0\geq 0$ if
$\int_{0}^x\frac{ds}{\eta(\phi^{-1}(s))}<\infty$,
\begin{align*}
p(t)&=G(c(t)) +\int_0^{\alpha(t)}g(t,s)ds,\\
\Psi(x)&=\int_{x_1}^x\frac{ds}{w(\phi^{-1}(G^{-1}(s)))},\ x>0,\
x_1>0.
\end{align*}
Here, $G^{-1}$ and $\Psi^{-1}$ are the inverse functions of $G$
and $\Psi$, respectively.
\end{thm}

\begin{rem}
We note that $\Psi$ is a strictly increasing function. Hence, if
$\Psi$ is unbounded we obviously have that
$$\Psi(p(t))+\int_0^{\alpha(t)}f(t,s)ds\in Dom(\Psi^{-1})$$ for all
$t\in\mathbb{R}^+_0$. Consider the case that Im$(\Psi)=(m,M)$,
where $m<\Psi(p(0))$ and $M=\sup\{\Psi(x):x\in(0,\infty)\}$. Let
us fix a number $\delta>0$ and consider $\tau\in(0,\delta)$ such
that
$$0\leq\int_0^{\alpha(t)}f(t,s)ds<M-\Psi(p(\delta)),\ t\in[0,\tau].$$
This number $\tau$ certainly exists since
$\int_0^{\alpha(t)}f(t,s)ds$ is a continuous function and
$\int_0^{\alpha(0)}f(0,s)ds=0$. From the above inequality we can
write
$$\Psi(p(t))+\int_0^{\alpha(t)}f(t,s)ds\leq\Psi(p(\delta))+\int_0^{\alpha(t)}f(t,s)ds<M,\ t\in[0,\tau],$$
that is,
$$\Psi(p(t))+\int_0^{\alpha(t)}f(t,s)ds\in Dom(\Psi^{-1}),\ t\in[0,\tau].$$
\end{rem}

\begin{proof}
Letting $t=0$ in (\ref{in1}), we observe that inequality
(\ref{in2}) holds trivially for $t=0$. Fixing an arbitrary number
$t_0\in(0,\tau]$, we define on $[0,t_0]$ a positive and
nondecreasing function $z(t)$ by
\begin{equation*}
z(t)=c(t_0)+\int_0^{\alpha(t)}\left[f(t_0,s)\eta(u(s))w(u(s))+g(t_0,s)\eta(u(s))\right]ds.
\end{equation*}
Then, $z(0)=c(t_0)$,
\begin{equation}
\label{in4} u(t)\leq\phi^{-1}(z(t)),\ t\in[0,t_0],
\end{equation}
and
\begin{align*}
z'(t)&=[f(t_0,\alpha(t))\eta(u(\alpha(t)))w(u(\alpha(t)))
+g(t_0,\alpha(t))\eta(u(\alpha(t)))]\alpha'(t)\\
&\leq\eta(\phi^{-1}(z(\alpha(t))))[f(t_0,\alpha(t))w(\phi^{-1}(z(\alpha(t))))
+g(t_0,\alpha(t))]\alpha'(t)\, .
\end{align*}
Since $\alpha(t)\leq t$, we deduce that
\begin{equation*}
\frac{z'(t)}{\eta(\phi^{-1}(z(t)))}\leq[f(t_0,\alpha(t))w(\phi^{-1}(z(\alpha(t))))
+g(t_0,\alpha(t))]\alpha'(t).
\end{equation*}
Integrating the above relation on $[0,t]$ yields
\begin{equation*}
G(z(t))\leq G(c(t_0)) +\int_0^{\alpha(t_0)}g(t_0,s)ds +
\int_0^{\alpha(t)}f(t_0,s)w(\phi^{-1}(z(s)))ds,
\end{equation*}
which implies that
\begin{equation}
\label{in3} z(t)\leq G^{-1}\left[p(t_0) +
\int_0^{\alpha(t)}f(t_0,s)w(\phi^{-1}(z(s)))ds\right] \, .
\end{equation}
Defining $v(t)$ on $[0,t_0]$ by
$$v(t)=p(t_0) +
\int_0^{\alpha(t)}f(t_0,s)w(\phi^{-1}(z(s)))ds,$$ we have that
$v(0)=p(t_0)$ and
\begin{align*}
v'(t)&=[f(t_0,\alpha(t))w(\phi^{-1}(z(\alpha(t))))]\alpha'(t)\\
&\leq[f(t_0,\alpha(t))w(\phi^{-1}(G^{-1}(v(\alpha(t)))))]\alpha'(t),
\end{align*}
\textrm{i.e.},
\begin{equation*}
\frac{v'(t)}{w(\phi^{-1}(G^{-1}(v(t))))}\leq
f(t_0,\alpha(t))\alpha'(t) \, .
\end{equation*}
Integrating this last inequality from $0$ to $t$, we obtain
$$\Psi(v(t))\leq\Psi(v(0))+\int_0^{\alpha(t)}f(t_0,s)ds,$$
and from this we get
\begin{equation}
\label{in5} v(t_0)\leq
\Psi^{-1}\left[\Psi(p(t_0))+\int_0^{\alpha(t_0)}f(t_0,s)ds\right].
\end{equation}
From (\ref{in4}), (\ref{in3}) and (\ref{in5}) we deduce that
$$u(t_0)\leq \phi^{-1}\left\{G^{-1}\left(\Psi^{-1}\left[\Psi(p(t_0))
+\int_0^{\alpha(t_0)}f(t_0,s)ds\right]\right)\right\}.$$ Since
$t_0\leq\tau$ is arbitrary we are done with the proof.
\end{proof}

\begin{thm}
\label{teor2} Let functions $f$, $g$, $\phi$, $c$, $\eta$, $w$,
$\alpha$, $u$, $G$ and $\Psi$ be as in Theorem~\ref{teor1}. If for
all $t\in\mathbb{R}_0^+$ the inequality
\begin{equation}
\label{in6} \phi(u(t))\leq
c(t)+\int_0^{\alpha(t)}f(t,s)\eta(u(s))w(u(s))ds+\int_0^t
g(t,s)\eta(u(s))w(u(s))ds
\end{equation}
holds, then
\begin{equation}
\label{in7} u(t)\leq
\phi^{-1}\left\{G^{-1}\left(\Psi^{-1}\left[\Psi(G(c(t)))
+\int_0^{\alpha(t)}f(t,s)+\int_0^{t}g(t,s)ds\right]\right)\right\}
\end{equation}
for all $t\in[0,\tau]$, where $\tau>0$ is chosen in such a way that
$$
\Psi(G(c(t)))
+\int_0^{\alpha(t)}f(t,s)+\int_0^{t}g(t,s)ds\in Dom(\Psi^{-1}) \, .
$$
\end{thm}

\begin{proof}
Letting $t=0$ in (\ref{in6}), we observe that inequality
(\ref{in7}) holds trivially for $t=0$. Fixing an arbitrary number
$t_0\in(0,\tau]$, we define on $[0,t_0]$ a positive and
nondecreasing function $z(t)$ by
\begin{equation*}
z(t)=c(t_0)+\int_0^{\alpha(t)}f(t_0,s)\eta(u(s))w(u(s))ds+\int_0^t
g(t_0,s)\eta(u(s))w(u(s))ds\, .
\end{equation*}
Then, $z(0)=c(t_0)$,
\begin{equation}
\label{in8} u(t)\leq\phi^{-1}(z(t)), \quad t\in[0,t_0] \, ,
\end{equation}
and, since $\alpha(t)\leq t$,
\begin{align*}
z'(t)&=f(t_0,\alpha(t))\eta(u(\alpha(t)))w(u(\alpha(t)))\alpha'(t)
+g(t_0,t)\eta(u(t))w(u(t))\\
&\leq\eta(\phi^{-1}(z(t)))[f(t_0,\alpha(t))w(\phi^{-1}(z(\alpha(t))))\alpha'(t)
+g(t_0,\alpha(t))w(\phi^{-1}(z(t))]\, ,
\end{align*}
\textrm{i.e.},
\begin{equation*}
\frac{z'(t)}{\eta(\phi^{-1}(z(t)))}\leq
f(t_0,\alpha(t))w(\phi^{-1}(z(\alpha(t))))\alpha'(t)+g(t_0,t)w(\phi^{-1}(z(t))
\, .
\end{equation*}
Integrating the above relation on $[0,t]$ yields
\begin{equation*}
G(z(t))\leq
G(c(t_0))+\int_0^{\alpha(t)}f(t_0,s)w(\phi^{-1}(z(s)))ds+\int_0^{t}g(t_0,s)w(\phi^{-1}(z(s))ds
,
\end{equation*}
which implies that
\begin{equation}
\label{in9} z(t)\leq G^{-1}(v(t)), \quad t\in[0,t_0]
\end{equation}
where $v(t)$ is defined by
$$v(t)=G(c(t_0))+\int_0^{\alpha(t)}f(t_0,s)w(\phi^{-1}(z(s)))ds
+ \int_0^{t}g(t_0,s)w(\phi^{-1}(z(s)))ds.$$ Note that
$v(0)=G(c(t_0))$ and
\begin{align*}
v'(t)&=[f(t_0,\alpha(t))w(\phi^{-1}(z(\alpha(t))))]\alpha'(t)
+g(t_0,t)w(\phi^{-1}(z(t)))\\
&\leq
w(\phi^{-1}(G^{-1}(v(t))))[f(t_0,\alpha(t))\alpha'(t)+g(t_0,t)],
\end{align*}
\textrm{i.e.},
\begin{equation*}
\frac{v'(t)}{w(\phi^{-1}(G^{-1}(v(t))))}\leq
f(t_0,\alpha(t))\alpha'(t)+g(t_0,t).
\end{equation*}
Integrating this last inequality from $0$ to $t$, we obtain
$$\Psi(v(t))\leq\Psi(v(0))+\int_0^{\alpha(t)}f(t_0,s)ds+\int_0^{t}g(t_0,s)ds,$$
from which we get
\begin{equation}
\label{in10} v(t_0)\leq
\Psi^{-1}\left[\Psi(G(c(t_0)))+\int_0^{\alpha(t_0)}f(t_0,s)ds+\int_0^{t_0}g(t_0,s)ds\right].
\end{equation}
From (\ref{in8}), (\ref{in9}) and (\ref{in10}), we deduce that
$$u(t_0)\leq \phi^{-1}\left\{G^{-1}\left(\Psi^{-1}\left[\Psi(G(c(t_0)))
+\int_0^{\alpha(t_0)}f(t_0,s)ds+\int_0^{t_0}g(t_0,s)ds\right]\right)\right\}.$$
Since $t_0$ is arbitrary, inequality (\ref{in7}) is true.
\end{proof}


\section{Corollaries}

Theorem~\ref{teor1} and Theorem~\ref{teor2} generalize previous
results in the literature \cite{lip,yuan}.

Let $m>n>0$ be some constants. Define $\phi(x)=x^m$,
$c(t)=c^{m/(m-n)}$, $c>0$, and $\eta(x)=m/(m-n)x^n$ for
$x\in\mathbb{R}_0^+$. Then,
$$G(x)=\int_0^x\frac{ds}{\eta(\phi^{-1}(s))}
=\frac{m-n}{m}\int_0^x\frac{ds}{s^{n/m}}=x^{(m-n)/m}.$$ We have
$\lim_{x\rightarrow\infty}G(x)=\infty$. Assume that $f(t,s),\
g(t,s)$ do not depend on the variable $t$. Finally, let $x_1=1$.
Then we have the following from Theorem~\ref{teor1}:

\begin{cor}\cite[Theorem~2.1]{yuan}
\label{cor:y} If
$$u^m(t)\leq
c^{m/(m-n)}+\frac{m}{m-n}\int_0^{\alpha(t)}\left[f(s)u^n(s)w(u(s))+g(s)u^n(s)\right]ds,\
t\in\mathbb{R}_0^+,$$ then
$$u(t)\leq
\left\{\Psi^{-1}\left[\Psi\left(c+\int_0^{\alpha(t)}g(s)ds\right)
+\int_0^{\alpha(t)}f(s)ds\right]\right\}^{1/(m-n)},\
t\in[0,\tau],$$ where
$$\Psi(x)=\int_{1}^x\frac{ds}{w(s^{1/(m-n)})},\ x>0,$$
and $\tau\in\mathbb{R}^+$ is chosen such that
$$\Psi\left(c+\int_0^{\alpha(t)}g(s)ds\right)+\int_0^{\alpha(t)}f(s)ds\in Dom(\Psi^{-1}),\ t\in[0,\tau].$$
\end{cor}

\begin{rem}
Setting $m=2$ and $n=1$ in Corollary~\ref{cor:y}, we obtain
Lipovan's Theorem~1 of \cite{lip}.
\end{rem}

Theorem~\ref{teor1} permits also to enunciate, for example, the
following corollary:

\begin{cor}
Let $\phi(x)=x^n$ and $\eta(x)=(x^n+1)\ln(x^n+1)$ with $n>0$.
Further, let $c(t)=c>0$ and assume that the functions $f$, $w$ and
$\alpha$ are as in Theorem~\ref{teor1}, with $g\equiv 0$. If
$$u^n(t)\leq
c+\int_0^{\alpha(t)}\left[f(t,s)(u^n(s)+1)\ln(u^n(s)+1)w(u(s))\right]ds,\
t\in\mathbb{R}_0^+,$$ then there exists a number
$\tau\in\mathbb{R}^+$ such that
$$\Psi(G(c))+\int_0^{\alpha(t)}f(t,s)ds\in Dom(\Psi^{-1}),\ t\in[0,\tau],$$
and
$$u(t)\leq
\left\{G^{-1}\left(\Psi^{-1}\left[\Psi(G(c))
+\int_0^{\alpha(t)}f(t,s)ds\right]\right)\right\}^{1/n},\
t\in[0,\tau],$$ where
$$G(x)=\int_{x_0}^x\frac{ds}{(s+1)\ln(s+1)}, \quad x\geq c>x_0>0,$$
and
$$\Psi(x)=\int_{1}^x\frac{ds}{w(e^{e^s\ln(x_0+1)}-1)^{1/n})},\quad x>0,$$
since
$$G^{-1}(x)=e^{e^x\ln(x_0+1)}-1,\quad x>0.$$
\end{cor}

Interesting corollaries can also be obtained from
Theorem~\ref{teor2}. For instance, if we define $\phi$, $c$,
$\eta$, $f$ and $g$ as in the beginning of this section, we are
lead to a result proved in \cite{yuan}:
\begin{cor}\cite[Theorem~2.2]{yuan}
If $$u^m(t)\leq
c^{m/(m-n)}+\frac{m}{m-n}\left(\int_0^{\alpha(t)}f(s)u^n(s)w(u(s))ds+\int_0^t
g(s)u^n(s)w(u(s))ds\right)$$ for all $t\in\mathbb{R}_0^+$, then
\begin{equation}\label{in11}u(t)\leq \left\{\Psi^{-1}\left[\Psi(c)
+\int_0^{\alpha(t)}f(t,s)+\int_0^{t}g(t,s)ds\right]\right\}^{1/(m-n)}
\end{equation}
for all $t\in[0,\tau]$, where $\tau>0$ is chosen such that
$$\Psi(c)
+\int_0^{\alpha(t)}f(t,s)+\int_0^{t}g(t,s)ds\in Dom(\Psi^{-1}).$$
Here $$\Psi(x)=\int_{1}^x\frac{ds}{w(s^{1/(m-n)})},\ x>0.$$
\end{cor}
\begin{rem}
In \cite[Theorem~2.2]{yuan} the author assumed that
$\lim_{x\rightarrow\infty}\Psi(x)=\infty$. That means that
inequality (\ref{in11}) is valid for all $t\in\mathbb{R}_0^+$.
\end{rem}

Following some techniques introduced in \cite{GronTS},
it is our belief that the results of Theorems~\ref{teor1} 
and \ref{teor2} admit a generalization for time scales.
This is under study and will be addressed 
in a forthcoming paper.



\end{document}